\documentclass[11pt,english]{article}
\usepackage[margin= 2.2 cm,bottom=23mm, top= 21mm]{geometry}

\usepackage{amsthm}
\usepackage{amsmath}
\usepackage{amssymb}
\usepackage{setspace}
\usepackage{mathtools}
\usepackage{verbatim}
\usepackage{csquotes}
\usepackage{graphicx}
\usepackage[hidelinks]{hyperref}
\usepackage{thm-restate}
\usepackage{cleveref}
\usepackage{graphicx}
\usepackage{enumitem}
\setlist[itemize]{leftmargin=*} %sets itemize indentation to 0
\setlist[enumerate]{leftmargin=*}
\usepackage{framed}
\usepackage{subcaption}
\usepackage[T1]{fontenc}
\usepackage{bbm}
\usepackage{mathdots}
\usepackage{xcolor}
\usepackage{diagbox}
\usepackage{colortbl}

\theoremstyle{plain}

\newtheorem{theorem}{Theorem}[section]

\newtheorem{proposition}[theorem]{Proposition}
\newtheorem{lemma}[theorem]{Lemma}
\newtheorem{corollary}[theorem]{Corollary}

\theoremstyle{definition}

\newtheorem{defn}[theorem]{Definition}
\newtheorem*{defn*}{Definition}

\newtheorem{problem}[theorem]{Problem}

\expandafter\def\expandafter\normalsize\expandafter{%
    \normalsize
    \setlength\abovedisplayskip{4pt}
    \setlength\belowdisplayskip{4pt}
    \setlength\abovedisplayshortskip{4pt}
    \setlength\belowdisplayshortskip{4pt}
}

\usepackage[square,sort,comma,numbers]{natbib}
\newcommand{\calL}{\mathcal{L}}

\newcommand{\calU}{\mathcal{U}}

\newcommand{\cupdot}{\mathbin{\mathaccent\cdot\cup}}

\newcommand{\ba}{\mathbf{a}}
\newcommand{\bb}{\mathbf{b}}
\newcommand{\bx}{\mathbf{x}}
\newcommand{\by}{\mathbf{y}}
\newcommand{\bz}{\mathbf{z}}

\newcommand{\bit}{\mathrm{bit}}
\newcommand{\Star}{\mathrm{Star}}
\newcommand{\tw}{\mathrm{tw}}
\newcommand{\ex}{\mathrm{ex}}

\title{\vspace{-0.8cm} The growth rate of multicolor Ramsey numbers of $3$-graphs}
\author{Domagoj Brada\v{c}\thanks{Department of Mathematics, ETH, Z\"urich, Switzerland. Research supported in part by SNSF grant 200021\_196965. Email: \textbf{\{domagoj.bradac, benjamin.sudakov\}@math.ethz.ch}.}
\and Jacob Fox\thanks{Department of Mathematics, Stanford University, Stanford, CA. Email: \textbf{jacobfox@stanford.edu}. Research
supported by NSF Awards DMS-1953990 and DMS-2154129.}
\and Benny Sudakov\footnotemark[1]}
\date{}

\begin{document}
    \maketitle
    \begin{abstract}
        The $q$-color Ramsey number of a $k$-uniform hypergraph $G,$ denoted $r(G;q)$, is the minimum integer $N$ such that any coloring of the edges of the complete $k$-uniform hypergraph on $N$ vertices contains a monochromatic copy of $G$. The study of these numbers is one of the most central topics in combinatorics. One natural question, which for triangles goes back to the work of Schur in 1916, is to determine the behavior of $r(G;q)$ for fixed $G$ and $q$ tending to infinity. In this paper we study this problem for $3$-uniform hypergraphs and determine the tower height of $r(G;q)$ as a function of $q$. More precisely, given a hypergraph $G$, we determine when $r(G; q)$ behaves polynomially, exponentially or double-exponentially in $q$. This answers a question of Axenovich, Gy\'{a}rf\'{a}s, Liu and Mubayi.
    \end{abstract}
    
    \section{Introduction}
    
    Given $k$-uniform hypergraphs, or $k$-graphs, $G_1, \dots, G_q,$ let $r(G_1, \dots, G_q)$ denote their Ramsey number, which is the minimum positive integer $N$ such that in every coloring of the edges of the complete $k$-graph $K_N^{(k)}$ on $N$ vertices with color set $[q]=\{1,\ldots,q\}$ there is a color $i$ for which there is a monochromatic copy of $G_i$ in color $i$. When $G_1 = \dots = G_q = G,$ we write $r(G; q)$ and when $G = K_n^{(k)},$ we sometimes write $r_k(n;q).$ The existence of these numbers was famously proved by Ramsey~\cite{ramsey} in 1930. Since then, obtaining good bounds on $r_k(G;q)$ for various (hyper)graphs $G$ has been among the most significant areas of study in discrete mathematics. One of the central problems in this area is to obtain good bounds on the so-called diagonal graph Ramsey number, $r_2(n; 2),$ for which the current best bounds are $\sqrt{2}^n < r (n; 2) \le (4 - \epsilon)^n,$ where the lower bound is due to Erd\H{o}s \cite{erdos} and the upper bound is a recent breakthrough of Campos, Griffiths, Morris and Sahasrabudhe \cite{CGMS}. For a survey on graph Ramsey numbers we refer the reader to~\cite{conlon2015recent}.

    Another classical direction in Ramsey theory is given a fixed graph $G,$ to determine the behavior of $r(G; q)$ as the number of colors, $q$, tends to infinity. In the case when $G$ is a triangle, the study of this problem goes back to the work of Schur in 1916, who proved a Ramsey-type result for sum-free sets (see \cite{NR}). For general $G$, this problem exhibits the following dichotomy. If $G$ is bipartite, then $r(G; q) = O(q^C)$ for some constant $C = C(G).$ Indeed, this follows from the famous theorem of K\"{o}v\'{a}ri, S\'{o}s and Tur\'{a}n~\cite{kovari-sos-turan} stating that for bipartite $G,$ there is a constant $\epsilon =\epsilon(G) > 0$ such that for large enough $n$, any graph on $n$ vertices with at least $n^{2-\varepsilon}$ edges contains a copy of $G.$ On the other hand, if $G$ is not bipartite, then we have $r(G; q) > 2^q.$ This follows by considering the $q$-edge-coloring of the complete graph on the vertex set $\{0, 1\}^q$ where a pair of vertices is colored by the index of the first coordinate in which their binary representations differ. In this coloring, every color class is a bipartite graph, so there is no monochromatic copy of $G$. Day and Johnson~\cite{day17} have improved this lower bound by showing that for any non-bipartite graph $G,$ there is a positive $\epsilon > 0$ such that $r(G; q) > (2 + \epsilon)^q.$ Regarding upper bounds, a simple extension of the neighbour chasing argument of Erd\H{o}s and Szekeres~\cite{erdos-szekeres} yields $r(K_n; q) < q^{nq}.$ Hence, for fixed non-bipartite $G,$ we have $(2 + \epsilon)^q \le r(G; q) \le 2^{O(q \log q)}.$ Determining whether these numbers should be exponential or not is a very old and major open problem even for the simplest case when $G = K_3$ for which Erd\H{o}s offered a prize of \$250~\cite{chung1998erdos}. This problem has an interesting connection to the celebrated Shannon capacity in information theory. Namely, the maximum possible Shannon capacity of a graph with independence number $t$ is equal to $\lim_{q \rightarrow \infty} r(K_{t+1}; q)^{1/q}$ (see e.g.~\cite{alon2020lovasz}).

    Although already for graph Ramsey numbers there are significant gaps between
the lower and upper bounds, our knowledge of hypergraph Ramsey numbers is even
weaker. In the clique case, Erd\H{o}s and Rado~\cite{erdos1952combinatorial} showed that for some constant $c = c(q,k),$ the Ramsey numbers satisfy $r_k(n; q) \le \tw_k(cn),$ where $\tw_k(x)$ denotes the tower function defined as $\tw_1(x) = x$ and $\tw_k(x) = 2^{\tw_{k-1}(x)}$ for $k\ge 2.$ On the other hand, an ingenious construction of Erd\H{o}s and Hajnal (see e.g. \cite{graham1991ramsey}), known as the stepping-up lemma, allows one to obtain a lower bound for hypergraphs of uniformity $k+1$ from lower bounds for uniformity $k,$ essentially gaining an extra exponential at every step. However, this construction only works if the number of colors, $q,$ is at least $4$ or the uniformity, $k$, is at least $3.$ Therefore, we have $r_k(n; 4) = \tw_k(\Theta(n))$ and the order of magnitude of $r_k(n; 2)$ depends on the behaviour of $3$-uniform case. 
    The question whether $r_3(n; 2)$ grows doubly-exponentially remains one of the most intriguing open problems.
    We refer the reader to the surveys \cite{conlon2015recent, mubayi2020survey} for more details about hypergraph Ramsey problems.
    
    The focus of this work is to determine the growth rate of $r(G; q)$ for fixed $G$ and $q$ tending to infinity. This is a natural variant of Erd\H{o}s' question (mentioned above) for hypergraphs. We say that a function $f(q)$ grows as a tower of height $h$ if $\tw_h(\Omega(q^c)) \le f(q) \le \tw_h(O(q^C))$ for some constants $c, C > 0.$ We study the following problem.

    \begin{problem} \label{prob:main}
        Given a fixed $k$-uniform hypergraph $G$, determine the integer $h$ (if it exists) such that $r(G; q)$ grows as a tower of height $h$ as $q$ tends to infinity.
    \end{problem}
    Clearly, not every function grows as a tower of some height, but it might be natural to guess that this is the case for $r(G; q)$ for any fixed $k$-uniform hypergraph $G.$ As discussed above, in the graph case we have that $r(G; q)$ grows as a tower of height $1$ if $G$ is bipartite (and has at least two edges) whereas otherwise it grows as a tower of height $2$. The $3$-uniform case was first studied almost 50 years ago by Abbott and Williams~\cite{abbott-williams} who, using a modification of the stepping-up construction showed that $r(K^{(3)}_4; q)$ grows as a tower of height $3$. The $3$-uniform case has been revisited in more depth recently by Axenovich, Gy\'{a}rf\'{a}s, Liu and Mubayi~\cite{axenovich}. They observed that $r(G;q)$ is at most polynomial, i.e. grows as a a tower of height $1$ in $q$ if and only if $G$ is tripartite and they determined several classes of $3$-graphs for which $r(G;q)$ grows as a tower of height $2$. Furthermore, they ask the following question.

    \begin{problem}[\cite{axenovich}] \label{prob}
        For which $3$-uniform hypergraphs $G$, is $r(G; q)$ double exponential? Are there other jumps that the Ramsey function exhibits?
    \end{problem}
    
    We resolve Problem~\ref{prob:main} in the case $k=3$ and answer the question of Axenovich, Gy\'{a}rf\'{a}s, Liu and Mubayi in following strong sense. We show that for every non-tripartite $3$-uniform hypergraph $G,$ either $2^{\Omega(q)} < r(G;q) < 2^{q^{C}}$ for some $C = C(G)$ or $2^{2^{q/2}} < R(G;q)$ and characterize which $3$-graphs have which behaviour.
        
    To state our main result formally, we first require a definition.
    \begin{defn} \label{def:reducible}
        Let $G$ be a $3$-graph. A set $U \subseteq V(G)$ with $2 \le |U| < |V(G)|$ is called \emph{collapsible} if no edge of $G$ intersects $U$ in exactly two vertices. Let $v^*$ denote a new vertex and let $H$ be the $3$-graph with vertex set $(V(G) \setminus U) \cup \{v^*\}$ and edge set $E(H) = \{ e \in E(G) \, \vert \, e \cap U = \emptyset\} \cup \{ xyv^* \, \vert \, \exists u \in U, xyu \in E(G)\}.$ We say that $H$ is obtained from $G$ by \emph{collapsing} $U$ and that $G$ is \emph{reducible} to the pair $(H, G[U])$ by collapsing $U$.
    \end{defn}

    We define a nested sequence of sets of $3$-graphs $\calU_0 \subseteq \calU_1 \subseteq \dots$ as follows. First, $\calU_0$ consists of all tripartite $3$-graphs. The set $\calU_1$ contains the $3$-graphs for which there is a subset of vertices intersecting every edge in exactly one vertex (note that $\calU_1 \supseteq \calU_0$). For $i > 1,$ $\calU_i$ is the maximal set containing $\calU_{i-1}$ and any hypergraph which is reducible to some $(H, F)$ with $H \in \calU_{i-1}, F \in \calU_i.$ Note that if $G$ is reducible to $(H, F),$ then by definition, $v(H), v(F) < v(G),$ implying that the sets $\calU_{i}$ are indeed well-defined. Let $\calU \coloneqq \bigcup_{i \ge 0} \calU_i.$

    We are ready to state our main result determining the behaviour of $r(G;q)$ for any fixed $3$-graph $G$.

    \begin{theorem} \label{thm:3-graph}
        Let $G$ be a fixed $3$-uniform hypergraph with at least two edges.
        \begin{enumerate}[label=\alph*)]
            \item \label{thm:tripartite} If $G \in \calU_0$ (i.e. $G$ is tripartite), then $r(G;q) = q^{\Theta(1)}.$
            \item \label{thm:3-graph-single} If $G \in \calU \setminus \calU_0,$ then $2^{\Omega(q)} \le r(G;q) \le 2^{q^{O(1)}}.$ More precisely, if $G \in \calU_\ell,$ then $r(G; q) \le 2^{O(q^\ell \log q)}.$ 
            \item \label{thm:double-exp} If $G \not\in \calU,$ then $2^{2^{q/2}} \le r(G;q) \le 2^{2^{O(q \log q)}}.$
        \end{enumerate}
    \end{theorem}

    Our characterization might seem a bit unwieldy at first, but it turns out to be convenient to work with. For example, using it we can show that most Steiner triple systems have double-exponential multicolor Ramsey numbers, but there are Steiner triple systems for which it is exponential.

    The rest of the paper is structured as follows. In the remainder of the introduction, we give some examples which might help understand the definition of the sets $\calU_i.$ We prove Theorem~\ref{thm:3-graph} in Section~\ref{sec:proofs} which is split into three subsections. In the first subsection we prove the upper bounds, starting with a sketch of the main ideas, in the second we prove the lower bounds and in the third we tie all the bounds together. In Section~\ref{sec:examples}, we provide examples of $3$-graphs exhibiting different behaviours of the multicolor Ramsey number. We finish with some concluding remarks in Section~\ref{sec:concluding}.

    We use standard notation throughout the paper. As it appears frequently in our proofs, we denote by $\Star^{(3)}(h)$ the $3$-graph on $h$ vertices with the edges being all triples containing a fixed vertex.

    \subsection{About the sets $\calU_i$}
    In this subsection, we briefly discuss the sets $\calU_i$ just defined. Observe first that for all $i \ge 0,$ the set $\calU_i$ is closed under taking subgraphs. The rest of the content of this subsection is not needed for any of our proofs, but it should help clarify the definitions and facilitate understanding the rest of the paper.
        
    First let us show that $\calU_i \setminus \calU_{i-1} \neq \emptyset$ for all $i \ge 1.$ It is easy to see that $\Star^{(3)}(4) \in \calU_1 \setminus \calU_0.$ Now, let $i \ge 1$ and suppose there is some $G_i \in \calU_i \setminus \calU_{i-1}$. We define $G_{i+1}$ as follows. The vertex set of $G_{i+1}$ is $\{x\} \cupdot A \cupdot B$ where $|A| = |B| = |V(G_i)|.$ Inside each of $A$ and  $B$ place a copy of $G_i$ and additionally let $G_{i+1}$ contain all $3$-edges of the form $\{x, a, b\}, a \in A, b \in B.$ See Figure~\ref{fig:Giplusone} for an illustration. First, let us show that $G_{i+1} \in \calU_{i+1}.$ Indeed, by collapsing $A,$ $G_{i+1}$ is reducible to $(H, G_i),$ where we shall describe $H$ shortly. Since $G_i \in \calU_i,$ to show that $G_{i+1} \in \calU_{i+1},$ it suffices to show that $H \in \calU_i$ as well. Indeed, $V(H) = \{ x, a\} \cup B$ where $a$ represents the collapsed set $A.$ Note that $H[B] \cong G_i$ and the remaining edges of $H$ are of the form $\{x, a, b\}$ with $b \in B.$ Hence, the set $B$ is collapsible in $H$ and thus $H$ is reducible to $(e, G_i),$ where $e$ denotes the $3$-graph consisting of a single edge. Since $e \in \calU_0$ and $G_i \in \calU_i,$ it follows that $H \in \calU_i,$ as claimed.
    
    Now, we show that $G_{i+1} \not \in \calU_i.$ We claim that any collapsible set in $G_{i+1}$ intersects at most one of $A, B.$ Indeed, suppose that $U$ is collapsible in $G_{i+1}$ and contains a vertex $a \in A$ and a vertex $b \in B.$ Since $\{x, a, b\} \in E(G_{i+1}),$ we must have $x \in U.$ However, since $\{x, a, b'\} \in E(G_{i+1})$ for any $b' \in B,$ it follows that $B \subseteq U$ and analogously $A \subseteq U,$ so $U = V(G_{i+1}),$ a contradiction. Now, suppose that $G_{i+1}$ is reducible to $(H, G[U])$ by collapsing some set $U.$ Without loss of generality, $U \cap B = \emptyset,$ so $G_i \cong G_{i+1}[B] \subseteq H,$ implying that $H \in \calU_i.$ As $U$ was arbitrary, we have that $G_{i+1} \not\in \calU_i.$

    On the other hand, for example the clique $K^{(3)}_4$ is not in $\calU.$ Indeed, it has no collapsible set and no set of vertices such that each edge contains precisely one vertex from this set. A slightly more complicated example is the $3$-graph $G$ depicted in Figure~\ref{fig:example}. Formally, we have $V(G) = \{ a, b, c, d, e\}$ and $E(G) = \{ acd, bcd, ace, bde, cde \}.$ It can be checked that $G \not\in \calU_1$ and that $\{a, b\}$ is the only collapsible set in $G$. However, the $3$-graph obtained by collapsing $\{a, b\}$ is isomorphic to $K_4^{(3)},$ so $G \not\in \calU.$

    \begin{figure}
        \centering
        \begin{minipage}{0.4\textwidth}
            \centering
            \includegraphics[height=5cm]{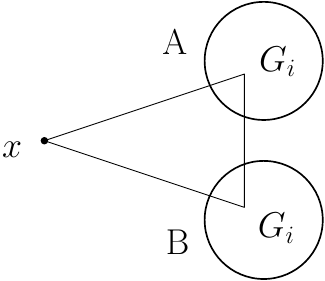} 
            \caption{Construction of $G_{i+1}$}
            \label{fig:Giplusone}
        \end{minipage}\hfill
        \begin{minipage}{0.5\textwidth}
            \centering
            \includegraphics[height=5cm]{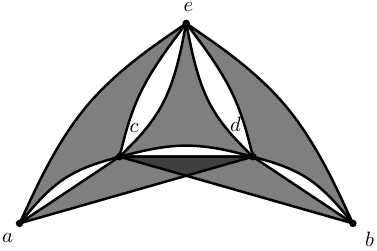}
            \caption{A $3$-graph not in $\calU$}
            \label{fig:example}
        \end{minipage}
    \end{figure}
        
    \section{Proof of Theorem~\ref{thm:3-graph}} \label{sec:proofs}
    \subsection{Upper bounds}
    \textbf{Proof sketch}\\
    This aim of this subsection is to prove the single-exponential upper bound in Theorem~\ref{thm:3-graph}~\ref{thm:3-graph-single}. Before presenting the proof formally, we illustrate our ideas on a simple example where $G$ is the Fano plane, that is, the unique $3$-graph on $7$ vertices with $7$ edges which all pairwise intersect in exactly one vertex.

    Let $U \subseteq V(G)$ denote the vertex set of an arbitrary edge in $G$. Note that by the abovementioned properties of the Fano plane, every edge intersects $U$ in either one or three vertices. Therefore, $U$ is collapsible and $G$ is reducible to the pair $(H, F)$ where $H = \Star^{(3)}(5),$ i.e. a $4$-clique in the link of a vertex, and $F$ is a single edge. Trivially, $F, H \in \calU_1$ which shows that $G \in \calU_2$. Though not required for the upper bound, it is easy to see that $G \not\in \calU_1.$

    Suppose we are given a $q$-colored complete $3$-graph $\Gamma$ on $N$ vertices, where $N$ is of the form $2^{O(q^2 \log q)}$ and we wish to show that there exists a monochromatic copy of $G.$ By considering all triples through a fixed vertex it is easy to see that $r(H; q) \le 1 + r(K^{(2)}_4; q) \le q^{4q}$ using the classical bound of Erd\H{o}s and Szekeres. Let $R = r(H; q).$ By definition, every set of $R$ vertices contains a monochromatic copy of $H,$ hence in $\Gamma$ there are at least $\binom{N}{R} / \binom{N-5}{R-5} \ge \frac{N^5}{R^5}$ monochromatic copies of $H$. By the pigeonhole principle, there is a set $S \subseteq V(\Gamma),$ $|S| = 4,$ and a colour, say red, such that there are at least $N / (q R^5)$ red copies of $H = \Star^{(3)}(5)$ with the set $S$ playing the role of the $4$-clique. Let $V'$ denote the set of vertices playing the role of the center of the star in these copies, so $|V'| \ge N / (q R^5).$

    Crucially, observe that if there is a red edge inside the set $V',$ then these three vertices along with the set $S$ contain a monochromatic copy of $G$ that we aim to find. Therefore, $V'$ is colored by $q-1$ colors. Iterating this argument inside $V'$, we see that it suffices to take $N \ge 3(q R^5)^q = 2^{O(q^2 \log q)},$ as claimed.

    For general $G,$ the argument is a little more complicated. Suppose that $G \in \calU_{\ell}$ and it is reducible to $(H, F)$ for some $H \in \calU_{\ell-1}, F \in \calU_{\ell}$ and let $U \subseteq V(G)$ be the collapsible set witnessing this. By a supersaturation argument analogous to the one above, we find a large set $V' \subseteq V(\Gamma)$ of vertices that can play the role of $v^* \in V(H)$ with the same set $S$ in the same color, say red. Then, however, inside the set $V'$, we obtain that there is no red copy of $F.$ In the case where $G$ is the Fano plane, $F$ is a single edge, which makes the argument simpler since we only need to ensure that $|V'| \ge r(G; q-1)$. In general, we shall require that $|V'|$ is at least the off-diagonal Ramsey number $r(F, G, G, \dots, G),$ where $G$ appears $q-1$ times.    

    We proceed with the formal proof. We start with the supersaturation argument outlined above, which allows us to reduce the target hypergraph in one of the colors.
    
    \begin{lemma} \label{lem:supersaturation-reduction}
        Let $G_1, \dots, G_q$ be given $3$-graphs. For $i \in [q],$ let $(H_i, F_i)$ be an arbitrary pair to which $G_i$ is reducible and if no such pair exists, let $H_i = G_i.$ Denoting $h = \max_{i \in [q]} v(H_i),$ we have
        \[ r(G_1, \dots, G_q) \le r(H_1, \dots, H_q)^h \cdot q \cdot \max\left\{ \{1\} \cup \{ r(G_1, \dots, G_{i-1}, F_i, G_{i+1}, \dots, G_q) \, \vert \, G_i \text{ is reducible} \} \right\}. \]
    \end{lemma}
    \begin{proof}
        For convenience, to each graph $H_i$ we add isolated vertices so that it has $h$ vertices which clearly does not change the value of $r(H_1, \dots, H_q).$ Denote $R \coloneqq r(H_1, \dots, H_q)$ and $N = r(H_1, \dots, H_q)^h \cdot q \cdot \max\left\{ \{1\} \cup \{ r(G_1, \dots, G_{i-1}, F_i, G_{i+1}, \dots, G_q) \, \vert \, G_i \text{ is reducible} \} \right\}$ and consider an arbitrary $q$-coloring of $K^{(3)}_N.$ Let $\Gamma$ denote this $q$-colored $3$-graph.
        
        By definition, any set of $R$ vertices of $\Gamma$ contains a copy of $H_i$ in color $i$ for some $i \in [q].$ Any such copy is contained in $\binom{N-h}{R-h}$ sets of $R$ vertices, so in total there are at least
        \[ \binom{N}{R} \big/ \binom{N-h}{R-h} \ge \frac{N^h}{R^h} \]
        distinct $h$-sets of $V(\Gamma)$ each of which is a monochromatic copy of $H_i$ in color $i$ for some $i \in [q].$ For such a copy, let $v^*, v_1, \dots, v_{h-1}$ denote its vertices with $v^*$ playing the role of the special vertex as in Definition~\ref{def:reducible} or an arbitrary vertex of $H_i$ if $H_i = G_i$. By the pigeonhole principle, there is a color $c \in [q]$ and an $(h-1)$-tuple of vertices $S = (w_1, \dots, w_{h-1})$ for which there are at least 
        \[ \frac{N^h}{R^h} / (q N^{h-1}) = \frac{N}{q R^h} \]
        copies of $H_c$ in color $c$ with $w_1, \dots, w_{h-1},$ in this order, playing the role of all vertices in of $H_i$ except $v^*$. If $H_c = G_c,$ we are done. Otherwise, let $V' \subseteq V(\Gamma)$ denote the set of vertices playing the role of $v^*$ in these copies, so $|V'| \ge \frac{N}{q R^h}.$
        
        Crucially, we claim that if there is a copy of $F_c$ in color $c$ inside $\Gamma[V'],$ this yields the desired copy of $G_c$ in color $c$ in $\Gamma$. Indeed, suppose there is such a copy in $V'$ and let $T \subseteq V'$ denote its vertex set. Let $U_c \subseteq V(G_c)$ be the collapsible set such that $G_c$ is reducible to $(H_c, F_c)$ by collapsing $U_c$ and let $V(G_c) \setminus U_c = \{x_1, \dots, x_m\}$. So by definition, $V(H_c) = \{v^*, x_1, \dots, x_m\}.$ Without loss of generality, we have for any $v \in V',$ the vertices $\{v, w_1, \dots, w_m\}$ form a copy of $H_c$ where $v$ is mapped to $v^*$ and $w_i$ is mapped to $x_i$ for every $i \in [m].$
        
        Then, $T \cup \{w_1, \dots, w_m\}$ forms a red copy of $G_c$ with $T$ being mapped to $U$ and $w_i$ being mapped to $x_i$ for $i \in [m].$ To see this, note first that by assumption, $T$ contains a red copy of $F_c = G_c[U].$ Furthermore, any edge $e = x_ix_jx_k \in E(G_c)$ disjoint from $U_c$ is contained in $H_c$ and since the vertices $v, w_1, \dots, w_m,$ for an arbitrary $v \in V',$ form a red copy of $H_c,$ it follows that the edge $w_iw_jw_k$ is red in $\Gamma$ as needed. Finally, by definition of a collapsible set, any other edge $e \in E(G_c)$ intersects $U_c$ in exactly one vertex. Consider such an edge $e = ux_ix_j$ with $u \in U_c.$ Then, we have $v^*x_ix_j \in E(H_c).$ Recall that $u \in V(F_c)$ so in the assumed red copy of $F_c,$ it is mapped to some vertex $v \in T \subseteq V'$. Since $v \in V',$ the vertices $v, w_1, \dots, w_m$ form a red copy of $H_c$ with $v$ mapped to $v^*$ and $w_i$ mapped to $x_i$ for $i \in [m].$ In particular, this implies that the edge $vw_iw_j$ is red in $\Gamma,$ as required.
        
        By our choice of $N,$ we have $|V'| \ge r(G_1, \dots, G_{c-1}, F_c, G_{c+1}, \dots G_q)$ so on $V'$ we either find a copy of $G_i$ in color $i$ for $i \in [q] \setminus \{c\}$ or a copy of $F_c$ in color $c,$ thus finishing the proof.
    \end{proof}
    
    To prove the upper bound in Part~\ref{thm:3-graph-single} of Theorem~\ref{thm:3-graph}, we use the preceding lemma and apply induction.
    
    \begin{lemma} \label{lem:explicit-upper-bound}
        Let $\ell \ge 1$ and let $G_1, \dots, G_q \in \calU_\ell$ be $3$-graphs each on at most $h$ vertices and denote $t = \sum_{i=1}^q v(G_i).$ Then, 
        \[ r(G_1, \dots, G_q) \le (qh)^{q^{\ell-1} \cdot h^{2\ell} t}. \]
    \end{lemma}
    \begin{proof}
        We prove the lemma by induction on $\ell, h, t.$ We assume $h \ge 3,$ otherwise there is nothing to prove. Consider first $\ell = 1$ and recall that by definition, each of the graphs $G_i$ has a subset of vertices $U_i$ intersecting every edge in precisely one vertex. For every $i$ for which $|U_i| > 1,$ let $(H_i, F_i)$ denote the resulting pair of graphs obtained by collapsing $U_i.$ Note that $F_i$ is the empty graph on $|U_i|$ vertices and $H_i$ is a subgraph of $\Star^{(3)}(v(G_i) - |U_i| + 1).$ If $|U_i| = 1,$ then let $H_i = G_i$ which is again a subset of a star $\Star^{(3)}(v(G_i) - |U_i| + 1)$. Consider a $q$-colored $3$-uniform clique. In order to find a copy of $\Star^{(3)}(s_i)$ in color $i$ for some $i,$ we can fix an arbitrary vertex $v$ and then in its link find a graph clique of size $s_i$ in color $i$. Thus, we can use a classical result in graph Ramsey theory, $r_2(n_1, \dots, n_q) < q^{\sum_{i=1}^q n_i},$ to obtain that $r(H_1, \dots, H_q) \le q^{\sum_{i=1}^q v(G_i) - |U_i| + 1}.$ Note that if $|U_i| > 1,$ then $r(G_1, \dots, G_{i-1}, F_i, G_{i+1}, G_r) \le v(F_i) \le h$ since $F_i$ has no edges. Applying Lemma~\ref{lem:supersaturation-reduction}, we obtain 
        \[ r(G_1, \dots, G_q) \le  (q^{\sum_{i=1}^q v(G_i) - |U_i| + 1})^h \cdot q \cdot h \le q^{ht+1} h \le (qh)^{h^2 t}, \]
        where in the last inequality we used $h \ge 3.$
        
        Now, let $\ell > 1$ and assume we have proved the statement for all sequences of graphs in $\calU_{\ell-1}$ as well as all sequences of graphs in $\calU_\ell$ each on at most $h$ vertices with in total at most $t-1$ vertices. Clearly, we may assume that $G_1 \in \calU_\ell \setminus \calU_{\ell-1}.$ For each $i \in [q]$ such that $G_i$ is reducible, let $(H_i, F_i)$ be a pair with $H_i \in \calU_{\ell-1}, F_i \in \calU_\ell$ to which $G_i$ is reducible and recall that $v(H_i), v(F_i) < v(G_i).$ For each $i$ such that $G_i$ is not reducible, let $H_i = G_i.$ Applying Lemma~\ref{lem:supersaturation-reduction} and the induction hypothesis, we have
        
        \begin{align*}
              r(G_1, \dots, G_q) &\le r(H_1, \dots, H_q)^h \cdot q \cdot \max\left\{ \{1\} \cup \{ r(G_1, \dots, G_{i-1}, F_i, G_{i+1}, \dots, G_q) \, \vert \, G_i \text{ is reducible} \} \right\} \\
                &\le \left((qh)^{q^{\ell-2} h^{2\ell-2} (t-1)}\right)^h \cdot q \cdot (qh)^{q^{\ell-1} h^{2\ell} (t-1)} \le (qh)^{q^{\ell-1} h^{2\ell} \cdot (\frac{t-1}{qh} + \frac{1}{qh} + t-1)} \le (qh)^{q^{\ell-1} h^{2\ell} \cdot t},
        \end{align*}
        where in the last inequality we used that $t \le qh.$
    \end{proof}

    Applying Lemma~\ref{lem:explicit-upper-bound} with $G_1 = \dots = G_q = G$ we obtain the upper bound claimed in Theorem~\ref{thm:3-graph},~Part~\ref{thm:3-graph-single}.
    \begin{corollary}\label{cor:upper-bound-Ui}
        If $G \in \calU_\ell,$ then $r(G;q) \le 2^{O(q^\ell \log q)}.$        
    \end{corollary}

    \subsection{Lower bounds}
    \begin{defn} \label{def:decomposable}
      Let $G$ be a $3$-uniform hypergraph. Suppose there is a partition of its vertex set $V(G) = V_1 \cupdot V_2 \dots \cupdot V_t$ with $|V_1|, t \ge 2$ such that for any edge $e \in E(G),$ and any $i \in [t],$ we have $|e \cap V_i| \neq 2.$ For $i \in [t],$ let $F_i \coloneqq G[V_i]$ and let $H$ be the $3$-uniform hypergraph obtained by collapsing each of the sets $V_i$ into a single vertex. Formally, $V(H) = [t]$ and $E(H) = \{ xyz \, \vert \, \exists e \in E(G), |e \cap V_x| = |e \cap V_y| = |e \cap V_z| = 1\}.$ We say that $G$ can be \emph{decomposed} into $(H; F_1, \dots, F_t).$ 
    \end{defn}

    In our proofs of the lower bounds, Definition~\ref{def:decomposable} will play a similar role that Definition~\ref{def:reducible} played in the proofs of the upper bounds. By taking $V_1 = U$ and $|V_2| = \dots = |V_t| = 1,$ informally speaking, we recover the definition of reducibility. On the other hand, a reduction with $t$ parts can, in some sense, be viewed as a sequence of at most $t$ simple reductions. Formally, we have the following lemma.

    \begin{lemma} \label{lem:same}
        If $G$ can be decomposed into $(H; F_1, \dots, F_t),$ where $H, F_1, \dots, F_t \in \calU,$ then $G \in \calU$.
    \end{lemma}
    \begin{proof}
        Let $V(G) = V_1 \cupdot \dots \cupdot V_t$ be the partition exhibiting that $G$ can be decomposed into $(H; F_1, \dots, F_t).$ Without loss of generality, assume that $|V_i| \ge 2$ for $i \in [s]$ and $|V_i| = 1$ for $s+1 \le i \le t.$ Denote $G_0 = G$ and for $i = 1, \dots, s,$ let $G_i$ be obtained from $G_{i-1}$ by collapsing $V_i.$ Note that these collapses are valid since a set $V_i$ remains collapsible after collapsing a disjoint set $V_j, j < i.$ The final graph $G_s$ is isomorphic to $H$, hence $G_s \in \calU.$ By definition, for each $0 \le i \le s-1,$ $G_i$ is reducible to $(G_{i+1}, G[V_{i+1}]),$ where $G[V_{i+1}] = F_{i+1} \in \calU$. Hence, by reverse induction, it follows that $G_{s-1}, \dots, G_0 = G$ are also in $\calU,$ as claimed.
    \end{proof}

    Our lower bound constructions are based on the stepping-up approach of Erd\H{o}s and Hajnal. First, we recall an important function used in this construction. For a nonnegative integer $x,$ let $x = \sum_{i=0}^{\infty} a_i 2^i$ be its unique binary representation (where $a_i = 0$ for all but finitely many $i$). We denote $\bit(x, i) \coloneqq a_i.$ Then for distinct $x, y \in \mathbb{Z}_{\ge 0},$ we define $\delta(x, y) \coloneqq \max \{ i \in \mathbb{Z}_{\ge 0} \, \vert \, \bit(x, i) \neq \bit(y, i)\}.$ See Figure~\ref{fig:deltas} for an illustration.

    \begin{figure}[h]
        \centering
        \includegraphics[scale=0.9]{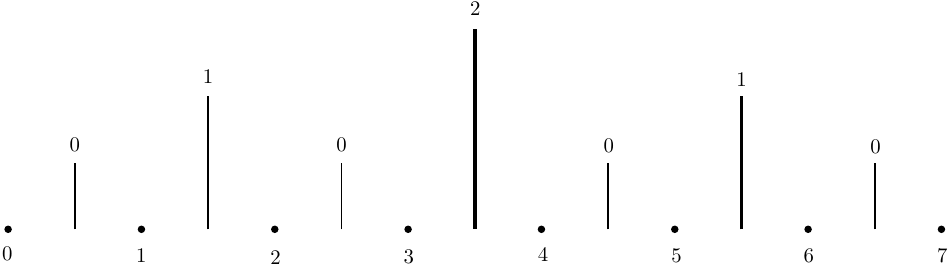}
        \caption{It is convenient to think about the function $\delta$ in the following way. The value of $\delta(x, y)$ is given by the highest line between $x$ and $y$ on the picture. So, for example, $\delta(0, 1) = \delta(6, 7) = 0,$ $\delta(0, 3) = \delta(5, 6) = 1, \, \delta(3, 4) = \delta(2, 7) = 2.$}
        \label{fig:deltas}
    \end{figure}

    The following properties of this function are well known and easy to verify.
    \begin{enumerate}[label=P\arabic*)]
        \item \label{prop:delta-smaller} $x < y \iff \bit(x, \delta(x, y)) < \bit(y, \delta(x, y))$.
        \item \label{prop:not-equal} For any $x < y < z,$ $\delta(x, y) \neq \delta(y, z)$.
        \item \label{prop:maximum} For any $x_1 < x_2 < \dots < x_k,$ $\delta(x_1, x_k) = \max_{1 \le i \le k-1} \delta(x_i, x_{i+1})$. 
    \end{enumerate}

    For every even $q$ we define a $q$-coloring $\phi_q$ of a complete $3$-uniform hypergraph on the vertex set $\{0, \dots, N_q-1\},$ where $N_q \coloneqq 2^{2^{q/2}}.$ For $0 \le x < y < z < N_q,$ let
    %and let $\delta_1 = \delta(x, y), \delta_2 = \delta(y, z).$ By~\ref{prop:not-equal}, we have $\delta_1 \neq \delta_2.$ Then, we set
    \[ \phi_q(x, y, z) \coloneqq  \left( \delta(\delta(x, y), \delta(y, z)), \mathbbm{1}\{ \delta(x,y) > \delta(y,z) \} \right). \]

    For example, $\delta(1, 4) = 2, \delta(4, 6) = 1$ and $\delta(2, 1) = 1,$ so $\phi_q(1, 4, 6) = (1, 1).$ 
    
    By~\ref{prop:not-equal}, we have $\delta(x, y) \neq \delta(y, z)$ so $\phi_q$ is well-defined. Additionally, note that $\delta(x,y), \delta(y,z) < 2^{q/2},$ implying $0 \le \delta(\delta(x,y), \delta(y,z)) \le q/2-1,$ so $\phi_q$ indeed uses at most $q$ colors.

    We are ready to prove our double-exponential lower bound.
    
    \begin{lemma} \label{lem:lb-double-exp}
        For any even $q,$ if $\phi_q$ contains a monochromatic copy of a $3$-graph $G,$ then $G \in \calU.$
    \end{lemma}
    
    \begin{proof}
        We prove the lemma using induction on $|V(G)|.$ For $|V(G)| < 3,$ there is nothing to prove.
        
        Now, consider a $3$-uniform hypergraph $G$ such that there is a monochromatic copy of $G$ in $\phi = \phi_q.$ Suppose the statement holds for all $3$-uniform hypergraphs with fewer vertices. Denote $N = N_q = 2^{2^{q/2}}.$ Suppose the color of this monochromatic copy is $(t, s),$ where $t \in \{0, \dots, q/2-1\}$ and $s \in \{0, 1\}.$  Let the vertices of $G$ be $\{1, \dots, h\}$ and without loss of generality, suppose that in the monochromatic copy vertex $i$ is embedded into $x_i$ where $0 \le x_1 < x_2 < \dots < x_h < N.$ 
        
        If $s = 1,$ for $i \in [h],$ define $x_i' = N-1 - x_i,$ i.e. $x_i'$ is obtained by complementing the binary representation of $x_i.$ Then, we have $0 \le x_h' < x_{h-1}' < \dots < x_1 < N$ and $\delta(x_i', x_j') = \delta(x_i, x_j)$ for any $1 \le i < j \le h.$ It follows that the set $\{x_h', \dots, x_1'\}$ forms a monochromatic copy of $G$ in color $(t, 0).$ Therefore, we may assume that $s = 0.$
        
        For $1 \le i < h,$ let $\delta_i \coloneqq \delta(x_i, x_{i+1}).$ Observe that by Property~\ref{prop:maximum}, we have
        \begin{equation} \label{eq:delta-xu-xv}
            \forall u,v, 1 \le u < v \le h, \delta(x_u, x_v) = \max_{u \le i < v} \delta_i.
        \end{equation}
        Let $m$ be the largest nonnegative integer such that $\bit(\delta_i, m)$ for $i \in [h-1]$ are not all equal. Since $\delta_1 \ne \delta_2,$ $m$ is well-defined. By Property~\ref{prop:maximum}, this choice of $m$ implies
        \begin{equation} \label{eq:bit-m}
            \forall 1 \le u < v \le h, \, \bit(\delta(x_u, x_v), m) = 1 \iff \exists i, u \le i < v, \bit(\delta_i, m) = 1.
        \end{equation}
        
        Suppose first that $m = t.$ Then,
        \begin{equation} \label{eq:m=t-edge}
            \forall u, v, w, u < v < w, uvw \in E(G) \implies \bit(\delta(x_u, x_v), m) = 0 \text{ and } \bit(\delta(x_v, x_w), m) = 1.             
        \end{equation}
        Indeed, this is true because for an edge $uvw,$ we have $\phi(x_ux_vx_w) = (m, 0).$

        Now, let $i$ be the minimal index such that $\bit(\delta_i, m) = 1.$
        Suppose that $i = h-1.$ Then by~\eqref{eq:delta-xu-xv}, for any $1 \le u < v \le h-1,$ we have $\bit(\delta(x_u, x_v), m) = \bit(\max_{u \le i < v} \delta_i, m) = 0.$ By~\eqref{eq:m=t-edge}, it follows that every edge of $G$ contains the last vertex $h,$ implying that $G \in \calU_1 \subseteq \calU.$
        
        Hence, we may assume that $i < h-1.$ Then, in $G$ there can be no edge $uvw$ with $u \le i$ and $v, w \ge i+1,$ as then $\bit(\delta(x_u, x_v), m) = 1$ by \eqref{eq:bit-m}, contradicting~\eqref{eq:m=t-edge}. Therefore, we can collapse the set $\{i+1, \dots, h\},$ which has at least two vertices by our assumption, to obtain a new $3$-uniform hypergraph $H$ on the vertex set $\{1, 2, \dots, i, v^*\}.$ Let us show that the vertex set $\{x_1, \dots, x_{i+1}\}$ forms a monochromatic copy of $H$ in color $(t, 0)$ with $j$ being embedded into $x_j$ for $j \in [i]$ and $v^*$ embedded into $x_{i+1}.$ Indeed, $\{x_1, \dots, x_i\}$ is a copy of $H[\{1, \dots, i\}]$ in color $(t, 0)$ because $\{x_1, \dots, x_h\}$ is a copy of $G$ in color $(t,0).$ Furthermore, for every edge $\{j, k, v^*\} \in E(H),$ we have $\bit(\delta(x_j, x_k), t) = 0$ and $\bit(\delta(x_k, x_{i+1}), t) = 1$ by our choice of $i$ and using Property~\ref{prop:maximum} so $\phi(x_j x_k x_{i+1}) = (t, 0)$. In $\phi,$ there clearly exists a monochromatic copy of the induced subgraph $G[\{v_{i+1}, \dots, v_h\}]$ so both $H$ and $G[\{v_{i+1}, \dots, v_h\}]$ are in $\calU$ by the induction hypothesis. It follows that $G \in \calU,$ as needed.
        
        Finally, suppose that $m \neq t.$ If $m < t,$ then by \eqref{eq:delta-xu-xv}, no edge is colored $(t, 0),$ so we assume $m > t.$ Let $1 \le i_1 < \dots < i_p < h$ denote all indices $i$ for which $\bit(\delta_i, m) = 1$ and note that $2 \le p+1 \le h.$ Let $I_1, \dots, I_{p+1}$ denote the intervals between consecutive $i_j's$. Formally, let $I_1 = \{1, \dots, i_1\},$ for $2 \le j \le p,$ let $I_j = \{i_{j-1}+1, \dots, i_j\}$ and let $I_{p+1} = \{i_p+1, \dots, h\}.$ 
        
        Suppose that there is an edge $e = uvw\in E(G)$ with $1 \le u < v < w \le h$ and $j \in [p+1]$ such that $|e \cap I_j| = 2.$ Since $I_j$ is an interval, we have either $e \cap I_j = \{u, v\}$ or $e \cap I_j = \{v, w\}.$ In the former case, by the definition of $I_j,$ using \eqref{eq:bit-m}, we have $\bit(\delta(x_u, x_v), m) = 0$ and $\bit(\delta(x_v, x_w), m) = 1,$ which implies $\phi(e) = (m, 0).$ Completely analogously, in the latter case we obtain $\phi(e) = (m, 1).$ Both cases contradict our assumptions, so we conclude that for any $e \in E(G)$ and $j \in [p+1],$ it holds that $|e \cap I_j| \neq 2.$
        
        For $j \in [p+1],$ denote $F_j = G[I_j].$ Furthermore, let $H$ be the hypergraph on the vertex set $\{1, \dots, p+1\}$ with edges $\{ uvw \, \vert \, \exists e \in G, \, |e \cap I_u| = |e \cap I_v| = |e \cap I_w| = 1\}.$ By definition, the hypergraphs $H, F_1, \dots, F_{p+1}$ have fewer vertices than $H.$ Hence, $G$ is decomposable into $(H; F_1, \dots, F_{p+1}).$ By the induction hypothesis, $F_1, \dots, F_{p+1} \in \calU$ since the vertices $\{ x_u \, \vert \, u \in I_j\}$ form a copy of $F_j$ in color $(t, 0)$ by assumption. For $j \in [p+1],$ let $y_j = x_{\min{I_j}}.$ Next we show that $\{y_1, \dots, y_{p+1}\}$ contains a monochromatic copy of $H$ in color $(t, 0)$. Indeed, consider the embedding which maps $i \in V(H) = [p+1]$ into $y_i.$ %Note that for every $u \in I_a, v \in I_b$ with $a < b,$ we have $\delta(x_u, x_v) = \max_{a \le j < v} \delta_j = \delta(y_u, y_v)$ by \eqref{eq:delta-xu-xv}.
        Consider an arbitrary edge $uvw \in E(H)$ with $1 \le u < v < w \le p+1.$ Recall that by definition there is a corresponding edge $abc \in E(G)$ with $a \in I_u, b \in I_v, c \in I_w.$ By~\eqref{eq:delta-xu-xv} and the definition of $i_1, \dots, i_p,$ we have
        \[ \delta(y_u,y_v) = \delta(x_{\min I_u}, x_{\min I_v}) = \max_{\min I_u \le j < \min I_v} \delta_j = \max_{u \le \ell < v} \delta_{i_\ell} = \delta(x_a, x_b), \]
        and analogously $\delta(y_vy_w) = \delta(x_b, x_c).$ Hence, $\phi(y_uy_vy_w) = \phi(x_ax_bx_c) = (t, 0).$ Thus the claimed embedding is indeed monochromatic so, by the induction hypothesis, we have $H \in \calU,$ which, using Lemma~\ref{lem:same} implies that $G \in \calU$ as well.
    \end{proof} 

    An exponential lower bound for non-tripartite $3$-uniform hypergraphs was proved in~\cite{axenovich}, but we include a proof for the sake of completeness.
   \begin{lemma} \label{lem:non-tripatite-lb}
        If $G$ is a non-tripartite $3$-uniform hypergraph, then $r(G;q) = 2^{\Omega(q)}.$
    \end{lemma}
    \begin{proof}
         Let $N = 2^{2q / 27}$ and consider $q$ random copies of the complete balanced tripartite $3$-uniform hypergraph, which has at least $\frac{2}{9} \binom{N}{3}$ edges, and define $\phi$ to be the coloring where each triple of $K^{(3)}_N$ is colored by the index of the first copy in which it appears. Since each color induces a tripartite graph, there is no copy of $G.$ It remains to show that with positive probability all edges are colored. Indeed, by a union bound, the probability that not all edges are colored is at most
        \[ \binom{N}{3} (1 - 2/9)^q < N^3 e^{-2q / 9} < 1, \]
        as needed.
    \end{proof}
        
    \subsection{Putting it together}
    \begin{proof}[Proof of Theorem~\ref{thm:3-graph}]        
        The lower bound in Part~\ref{thm:tripartite} is obtained by coloring edges of a complete $3$-uniform hypergraph on $\Omega(q^{1/3})$ vertices into distinct colors. For the upper bound, if $G$ is tripartite, by a well known result of Erd\H{o}s~\cite{erdos-hypergraph-kst}, there is an $\varepsilon > 0$ such that for large enough $N$, any $3$-uniform hypergraph on $N$ vertices with at least $N^{3-\varepsilon}$ edges, contains a copy of $G$. Hence, if we are given a $q$-colored complete graph on $N = (10q)^{1/\varepsilon}$ vertices, one of the colors will have at least $\binom{N}{3} / q > N^{3-\varepsilon}$ edges and thus contains a copy of $G$.

        The lower bound in Part~\ref{thm:3-graph-single} is given by Lemma~\ref{lem:non-tripatite-lb} and the upper bound in Corollary~\ref{cor:upper-bound-Ui}.

        Finally, the lower bound in Part~\ref{thm:double-exp} is given by Lemma~\ref{lem:lb-double-exp}, while the upper bound follows from the upper bound for cliques proved by Erd\H{o}s and Rado~\cite{erdos1952combinatorial}.
    \end{proof}

\noindent{\bf Remark.}
    If $G$ is tripartite and has at least two edges, its multicolor Ramsey number $r(G;q)$ is given by its extremal (or Tur\'{a}n) number $\ex(N, G)$ up to a logarithmic factor in the number of colors. Indeed, every color class in the Ramsey coloring has at most $\ex(N, G)$ edges, which implies that $q \geq 
    \Theta( N^3/\ex(N, G))$. On the other hand, by taking $q=O(\log N \cdot N^3/\ex(N, G) )$ random copies of an extremal $3$-uniform hypergraph on $N$ vertices and using similar computations as in Lemma~\ref{lem:non-tripatite-lb}, one can obtain a coloring with no monochromatic copy of $G$.

    \section{Examples} \label{sec:examples}
    Recall that for non-tripartite $G \in \calU,$ we have the lower bound $r(G; q) \ge 2^{\Omega(q)}$ given by Lemma~\ref{lem:non-tripatite-lb} while the upper bound is of the form $2^{O(q^\ell \log q)}$ for some $\ell \ge 1.$ Unless $G \in \calU_1,$ these bounds are far apart. However, in certain cases we can refine the lower bound. We start with a definition.

    \begin{defn}
        We say that a $3$-uniform hypergraph $G$ is \emph{forward-colorable} if there is a vertex partition $V_1 \cupdot \dots \cupdot V_t = V(G)$ such that for any edge $e \in E(G),$ there are $i < j$ for which $|e \cap V_i| = 1$ and $|e \cap V_j| = 2.$
    \end{defn}

    Observe that $\calU_2$ contains all forward-colorable $3$-uniform hypergraphs. Indeed, suppose $G$ is forward colorable with a vertex partition $V_1 \cupdot \dots \cupdot V_t$ as defined above. If $t = 2,$ every edge of $G$ touches $V_1$ in exactly one vertex, so $G \in \calU_1.$ Else, $U = V_1 \cup V_2$ is a collapsible set and $G$ is reducible to the pair $(H, G[U])$ where $H$ is forward-colorable with $t-1$ parts and $G[U] \in \calU_1.$ The claim follows by induction on $t$.
    
    Let $\calL_1$ be the maximal family containing all forward-colorable $3$-uniform hypergraphs as well as any $3$-uniform hypergraph which is reducible to some $(H; F_1, \dots, F_t)$ such that $H$ is tripartite and $F_1, \dots, F_t \in \calL_1.$
    \begin{lemma} \label{lem:rsquared-lb}
        For any $3$-uniform hypergraph $G$ not in $\calL_1,$ it holds that $r(G; q) \ge 2^{\Omega(q^2)}.$
    \end{lemma}
    \begin{proof}
        Let $q$ be a large integer and let $\phi$ be a coloring of $K^{(3)}_N$ with colors $\{1, \dots, q\}$ containing no monochromatic non-tripartite graph given by Lemma~\ref{lem:non-tripatite-lb}, where $N = 2^{\Omega(q)}.$ We define a coloring $\phi'$ on $N^q$ vertices using $3q$ colors and containing no monochromatic copy of any $3$-uniform hypergraph in $\calL_1$, the existence of which implies the statement. To describe $\phi',$ we identify the vertex set $[N^q]$ with $[N]^q.$ For a vector $\ba \in [N]^q$ we write $\ba = (\ba^1, \dots, \ba^q).$ Consider three vectors $\bx, \by, \bz \in [N]^q$ where $\bx < \by < \bz$ according to the lexicographic ordering which is defined as $\ba < \bb$ if for some $i \in [q]$, $\ba^i < \bb^i$ and $\ba^j = \bb^j$ for all $1 \le j < i$. Let $j$ be the first coordinate for which $\bx^j, \by^j, \bz^j$ are not all equal. If $\bx^j, \by^j, \bz^j$ are all distinct, then set $\phi'(\bx, \by, \bz) = \phi(\bx^j, \by^j, \bz^j).$ Else if, $\bx^j < \by^j = \bz^j,$ set $\phi'(\bx, \by, \bz) = (j, 0)$ and if $\bx^j = \by^j < \bz^j,$ then set $\phi'(\bx, \by, \bz) = (j, 1).$ Note that this  covers all cases by the assumed ordering.
        
        Now, we prove, by induction on $|V(G)|,$ that $\phi'$ is a Ramsey-coloring for any $3$-uniform hypergraph $G \not\in \calL_1.$ Let $G$ be a $3$-graph, denote $V(G) = \{1, \dots, h\}$ and suppose in $\phi'$ there exists a monochromatic copy of a $G$ with vertex $v \in [h]$ embedded into $\bx_v \in [N]^q.$ Assume the color of this copy is $(j, 0)$ or $(j, 1),$ for some $j \in [r].$ For $s \in [N],$ set $V_s = \{ v \in V(G) \, \vert \, \bx_v^j = s\}$. Then if the color of the copy is $(j, 0),$ it is easy to see that $G$ is forward-colorable with vertex partition $V_1 \cupdot \dots \cupdot V_N$ while if the color is $(j, 1)$, then $G$ is forward colorable with vertex partition $V_N \cupdot \dots \cupdot V_1.$ Thus in either case, we have $G \in \calL_1$. Now suppose the color of this monochromatic copy is $c \in [q].$ Let $j$ be the first coordinate in which $\bx_1, \dots, \bx_h$ are not all equal. Then, there is a partition of the vertex set $V(G) = V_1 \cupdot \dots \cupdot V_m,$ into $m \ge 2$ non-empty sets such that the vertices $V_i$ correspond to vectors with the same $j$-th coordinate. Let $H$ be the hypergraph with vertex set $[m]$ and edge set $E(H) = \{ abc \, \vert \, E(G) \cap (V_a \times V_b \times V_c) \neq \emptyset \}.$ It is easy to see that there is a monochromatic copy of $H$ in $\phi,$ and hence $H$ is tripartite. Additionally, for all $j \in [m],$ there trivially exists a monochromatic copy of $G[V_j]$ in $\phi'$ and hence $G[V_j] \in \calL_1$ by the induction hypothesis. It follows that $G \in \calL_1,$ as required.
    \end{proof}

    \begin{proposition} \label{prop:example-q-squared}
        There is a $3$-uniform hypergraph $G$ for which $r(G; q) = 2^{q^{2 + o(1)}}.$ 
    \end{proposition}
    \begin{proof}
        Let $G$ be the $3$-uniform hypergraph obtained by blowing up a non-central vertex of $\Star^{(3)}(4)$ by a set $A$ of $4$ vertices and placing a copy of $\Star^{(3)}(4)$ inside $A$. Let $v, a_1, a_2, a_3$ denote the vertices of $A$ with $v$ being the center and let $u, b_1, b_2$ denote the remaining vertices with $u$ being the center. 
        
        By collapsing the set $A$ we see that $G$ is reducible to $(\Star^{(3)}(4), \Star^{(3)}(4))$ implying that $G \in \calU_2$ and thus the upper bound follows by Corollary~\ref{cor:upper-bound-Ui}.
        
        Next we show that $G \not\in \calL_1$ and then the lower bound follows from Lemma~\ref{lem:rsquared-lb}. 
        %Note if $S$ is a subset $V(\Star^{(3)}(4))$ such that any edge of $\Star^{(3)}(4)$ contains either $1$ or $3$ vertices of $S,$ then either $|S| = 1$ or $|S| = 4.$ 
        First, suppose that $G$ is forward-colorable and let $V(G) = V_1 \cupdot \dots \cupdot V_t$ be a partition which certifies it. Then there are indices $i < j$ such that $v \in V_i$ and $\{a_1, a_2, a_3\} \subseteq V_j.$ By the same argument, since $\{u, v, b_1, b_2\}$ form a $\Star^{(3)}(4)$ with center $u,$ we have that $b_1, b_2 \in V_i$ and $u \in V_\ell$ for some $\ell < i.$ But, then the edge $ub_1a_1$ has its vertices in three distinct sets, a contradiction.
        
        Now, suppose that $G$ is decomposable into $(H; F_1, \dots, F_t)$ with a partition $V(G) = V_1 \cupdot \dots \cupdot V_t.$
        Note that if $S$ is a nonempty subset of $V(\Star^{(3)}(4))$ such that any edge of $\Star^{(3)}(4)$ contains either $1$ or $3$ vertices of $S,$ then either $|S| = 1$ or $|S| = 4.$ Suppose that some $V_i$ contains at least two vertices from  $v, u, b_1, b_2$. Since these vertices form a star, by the previous observation, it follows that $V_i$ contains all of them. Furthermore, since any $w \in A$ forms a copy of $\Star^{(3)}(4)$ with $\{u, b_1, b_2\},$ by the same observation, we get $V_i = V(G),$ a contradiction. Therefore, the vertices $u, v, b_1, b_2$ are in different sets, implying that $\Star^{(3)}(4) \subseteq H.$ Since $\Star^{(3)}(4)$ is not tripartite, it follows that $G \not\in \calL_1,$ as claimed.
    \end{proof}
    
    Let $G^{(3)}(n, p)$ denote the random $3$-uniform hypergraph on $n$ vertices where each hyperedge is included independently with probability $p$.

    \begin{proposition} \label{prop:random}
        There is a positive constant $C$ such that if $p \ge \frac{C}{n^2},$ then for $G \sim G^{(3)}(n, p),$ with high probability, we have $r(G; q) \ge 2^{2^{q/2}}.$
    \end{proposition}    
    \begin{proof}
        Using a standard Chernoff bound (see e.g.~\cite{alon-spencer}), it is easy to show that with high probability, 
        \begin{equation} \label{eq:edge-distr}
            |E(G) \cap (A_1 \times A_2 \times A_3)| \ge Cn / 10^9, \forall A_1, A_2, A_3 \subseteq V(G), |A_i| \ge n/100, \forall i \in[3].
        \end{equation}
        Conditioning on \eqref{eq:edge-distr}, we show that $G \not\in \calU,$ which would complete the proof by Lemma~\ref{lem:lb-double-exp}.

        Let us first informally explain the ideas of the proof. If $G \in \calU,$ then $G \in \calU_1$ or there is a collapsible set $U \subseteq V(G)$ such that $G$ is reducible to $(H, G[U])$ by collapsing $U$, where $H, G[U] \in \calU.$ If $|U| < n/2,$ next consider the hypergraph $G_2 = H$ and otherwise we ``put aside'' the vertices $V(G) \setminus U$ and consider the hypergraph $G_2 = G[U].$ Note that this way, $|V(G_2)| \ge |V(G)| / 2.$ By assumption, we have $G_2 \in \calU$ so we can apply the same reasoning as above. In general, at each step we have a hypergraph $G_i$ whose each vertex corresponds to a collapsed set or a single vertex in $G.$ Now, suppose that at some point we have in total put aside a set $T$ of at least $n / 100$ vertices. Since we never put aside more than half of the current number of vertices, we have $|T| < 0.99n$ so by~\eqref{eq:edge-distr}, in $G$ there is an edge with two vertices in $V(G) \setminus T$ and one vertex in $T.$ However this contradicts the fact that we only put aside vertices outside some collapsible set.
        
        Similarly, we can show that no vertex in $V(G_i)$ represents a set of more than $n / 100$ vertices of $G.$ Indeed, if in some step we collapse a set $U \subseteq V(G_i)$ representing in total at least $n / 100$ vertices of $G$ but no more than $0.99n$, by \eqref{eq:edge-distr}, in $G$ there is an edge with two vertices represented by $U$ and one vertex not represented by $U,$ a contradiction.
        
        On the other hand, if no vertex of $G_i$ represents more than $n / 100$ vertices, we can group the vertices of $G_i$ into four sets, where each set represents a set of at least $n / 100$ vertices of $G,$ which, by \eqref{eq:edge-distr}, implies that $G_i \not\in \calU_1.$ Therefore, for any $i,$ $G_i$ we can define a new hypergraph $G_{i+1}$ as above. However, clearly this process cannot go on indefinitely, which will yield a contradiction.

        We proceed to the formal proof. For the sake of contradiction, suppose $G \in \calU.$ Now, we run the following algorithm in steps $i=1,\dots$ At each step, we have a set $T_i \subseteq V(G),$ and a hypergraph $G_i,$ where each vertex $v \in V(G_i)$ is labelled with a set $S_i(v) \subseteq V(G)$ such that the sets $(S_i(v))_{v \in V(G_i)}$ partition $V(G) \setminus T_i.$ The hypergraph $G_i$ will correspond to a hypergraph obtained from $G$ after several reductions and a set $S_i(v)$ indicates that $v$ is a vertex representing the collapsed set (possibly in more than one step) $S_i(v)$. Formally, we always have 
        \begin{equation} \label{eq:G-def}
             E(G_i) = \{ v_1v_2v_3 \, \vert \, \exists e \in E(G), |e \cap S_i(v_j)| = 1, \forall j \in [3]\}.
        \end{equation}
        For $U \subseteq V(G_i),$ we denote $S_i(U) = \bigcup_{v \in U} S_i(v)$ and we denote its \emph{weight} by $w_i(U) = |S_i(U)|.$  We shall maintain the following:
        
        \begin{enumerate}[label=(\roman*)]
            \item $G_i \in \calU.$ \label{G_i-in-U}
            \item For any $v \in V(G_i), w_i(\{v\}) < n/100.$ \label{small-weights}
            \item $|T_i| < n/100$ and for any $e \in E(G), |e \cap T_i| \neq 1.$ \label{small-T}
            \item For any $e \in E(G)$ and any $v \in V(G_i),$ it holds that $|e \cap S_i(v)| \neq 2$. \label{non-degenerate}
        \end{enumerate}
               
       Initially, we set $G_1 = G,$ $S_1(v) = \{v\}, \forall v \in V(G)$ and $T_1 = \emptyset.$ Then, we proceed in steps $i=1,\dots$ as follows.

        By assumption, $G_i \in \calU.$ Suppose first that $G_i \in \calU_1,$ that is, there is a subset $W \subseteq V(G_i)$ such that any edge in $G_i$ intersects $W$ in exactly one vertex. Hence, either $W$ or $V(G_i) \setminus W$ is an independent set in $G_i$ with weight at least $n / 4.$ Let $I$ denote this independent set. Since $w(\{v\}) < n/100$ for any $v \in V(G_i),$ $I$ can be partitioned into three sets $A_1, A_2, A_3,$ with $w(A_i) \ge n/100,$ for all $i \in [3].$ However, by definition of $G_i$, this implies $E(G) \cap (A_1 \times A_2 \times A_3) = \emptyset,$ contradicting \eqref{eq:edge-distr}.

        Hence, $G_i \not\in \calU_1,$ implying that there is a collapsible subset $U_i \subseteq G_i$ such that $G_i[U_i] \in \calU$ and the hypergraph $H$ obtained by collapsing $U_i$ is also in $\calU.$ We consider two cases. 
        
        First, suppose that $w_i(U_i) \le n/2.$ Let us show that then $|w_i(U_i)| < n/100.$ Otherwise by \eqref{eq:edge-distr}, $G$ has an edge in $S_i(U_i) \times S_i(U_i) \times S_i(V(G) \setminus (T_i \cup U_i)).$ Such an edge cannot have two vertices in the same set $S_i(v)$ by Property~\ref{non-degenerate}. On the other hand, if all three of its vertices lie in different sets $S_i(v),$ this contradicts that $U_i$ is collapsible in $G_i,$ so indeed we have $|w_i(U_i)| < n/100.$ Now, we let $G_{i+1}$ be the hypergraph obtained from $G_i$ by collapsing $U_i$ and let $T_{i+1} = T_i.$ For any $v \in V(G_i) \setminus U_i,$ we let $S_{i+1}(v) = S_i(v)$ and for the new vertex $v^* \in V(G_{i+1})$ representing the collapsed set $U_i,$ we let $S_{i+1}(v^*) = \cup_{v \in U_i} S_i(v).$ Let us verify that Propeties~\ref{G_i-in-U}--\ref{non-degenerate} for $i+1.$ Property~\ref{G_i-in-U} holds by assmption, \ref{small-weights} still holds because $w_i(U_i) < n/100,$ \ref{small-T} is immediate since $T_{i+1} = T_i$ and finally, Property~\ref{non-degenerate} holds since $U_i$ is a collapsible set in $G_i.$

        Secondly, suppose that $w_i(U_i) > n/2.$ Denote $T_{i+1} = T_i \cup S_i(V(G_i) \setminus U_i),$ let $G_{i+1} = G_i[U_i]$ and $S_{i+1}(v) = S_i(v)$ for all $v \in U_i.$ Let us verify the invariants. Property~\ref{G_i-in-U} is given by the assumption, while properties \ref{small-weights} and \ref{non-degenerate} are immediate since $S_{i+1}(v) = S_i(v)$ for all $v \in U_i = V(G_{i+1}).$ Let us check Property~\ref{small-T}. Suppose first there is an edge $e \in E(G)$ such that $|e \cap T_i| = 1.$ Then, it has two vertices inside $S_i(U_i)$ and by Property~\ref{non-degenerate}, these two vertices are in distinct sets $S_i(v), S_i(v').$ However, this contradicts the fact that $U_i$ is collapsible in $G_i,$ proving the second part of~\ref{small-T}. Finally, we show that $|T_{i+1}| < n/100.$ Suppose otherwise. Recall that $G$ has no edges touching $T_i$ in exactly one vertex. Since $U_i$ is collapsible in $G_i,$ it follows that $G$ has no edges touching $T_{i+1}$ in exactly one vertex either. However, we have that $n/100 \le |T_{i+1}| \le n/2,$ which yields a contradiction to \eqref{eq:edge-distr} by taking the sets $V(G) \setminus T_i, V(G) \setminus T_i, T_i.$

        To conclude, in each step $i=1,\dots$ we obtain a new hypergraph $G_{i+1}$ still satisfying all the invariants. However, we always have $|V(G_{i+1})| < |V(G_i)|$ so the process cannot run indefinitely, a contradiction.                
    \end{proof}
    We remark that considering the process of collapsing sets is in some sense necessary in the proof above. Indeed, one might hope to prove Proposition~\ref{prop:random} by finding a fixed hypergraph $H \not\in \calU$ such that $H$ appears in $G \sim G^{(3)}(n, C/n^2)$ with high probability. This is however not possible. Indeed, for any fixed $k,$ the expected number of sets of $2k$ vertices spanning at least $k$ edges is $O(n^{2k} p^k) = O(C^k).$ By the Poisson paradigm, it follows that with probability $\Omega(1),$ $G$ does not have $2k$ vertices spanning at least $k$ edges for any fixed $k$. Thus, every subgraph $H$ of $G$ on at most $2k$ vertices has an edge whose all but at most one vertex has degree one in $H$. Collapsing this edge we get a new hypergraph, again having ratio less than $1/2$ between number of edges and vertices and therefore we can continue collapsing. This implies that with positive probability $G$ does not contain a fixed hypergraph not in $\calU.$ 
    
    Note that the only property of the random $3$-uniform hypergraph we used in the proof of Proposition~\ref{prop:random} is \eqref{eq:edge-distr}, i.e. that for any three sets of size at least $n/100,$ there is an edge with a vertex in each of the sets. The same property holds for most Steiner triple systems. This was proven in a stronger form implicitly by Kwan~\cite{kwan2020almost} and later stated by Ferber and Kwan~\cite[Theorem~8.1]{ferber2020almost}. Therefore we obtain the following corollary.
    \begin{corollary}
     A random  Steiner triple systems with high probability has double-exponential multicolor Ramsey numbers.
    \end{corollary}
   
    However, this is not the case for all Steiner triple systems. Indeed, let $m \ge 2,$ and consider the Steiner triple system $G$ on the vertex set $V(G) = \mathbb{F}_2^m \setminus \{0 \}$ where a triple $\bx\by\bz$ forms an edge if and only if $\bx + \by + \bz = 0.$ For $i \in [m],$ let $V_i$ be the set of vectors in $V(G)$ whose last $1$-coordinate is in the $i$-th place. The partition $V(G) = V_1 \cupdot V_2 \cupdot \dots \cupdot V_m$ shows that $G$ is forward-colorable, and hence $r(G; q) \le 2^{O(q^2 \log q)}$ by the upper bound in Theorem~\ref{thm:3-graph}~part~\ref{thm:3-graph-single}.

    \section{Concluding remarks} \label{sec:concluding}
    In this paper we determined, for any fixed $3$-uniform hypergraph $G,$ the tower height of its multicolor Ramsey number $r(G;q)$ as the number of colors tends to infinity. Several natural questions remain. The most obvious one is to resolve Problem~\ref{prob:main} for higher uniformities. We tentatively conjecture that the multicolor Ramsey number of any fixed uniform hypergraph grows as a tower of some height. A counterexample would be very interesting.

    Our methods do not seem to provide tight bounds for larger uniformities. For example, we do not know the correct answer even for the following $4$-uniform hypergraph: let $G$ be the $4$-uniform hypergraph with vertex set $A \cup B$ where $A, B$ are disjoint sets of some fixed size $t \ge 3$ and where a $4$-tuple forms an edge if and only if it intersects $A$ and $B$ in two vertices each. Since $G$ is not $4$-partite, $r(G; q)$ is at least exponential in $q$ as shown in \cite{axenovich} and we can show that $r(G; q)$ is at most double-exponential.

    For $3$-uniform hypergraphs $G \in \calU,$ our upper and lower bounds usually have different powers of $q$ in the exponent. It would be interesting to refine these bounds further. A natural simple example is the Fano plane for which we have $2^{\Omega(q)} \le r(\mathrm{Fano}; q) \le 2^{O(q^2 \log q)}$.

    It is easy to see that $r(\Star^{(3)}(4); q) = 2^{q^{1 + o(1)}}$ and Proposition~\ref{prop:example-q-squared} provides a $3$-uniform hypergraph $G$ with $r(G; q) = 2^{q^{2+o(1)}}.$ However, for each $\ell \ge 3,$ there are $3$-uniform hypergraphs $G_\ell$ for which our best upper bound is of the form $r(G_\ell; q) \le 2^{q^{\ell+o(1)}}.$ It would be interesting to determine whether this can be tight.

    \begin{problem}
        Does there exist, for every $\ell \ge 1,$ a $3$-uniform hypergraph $G_\ell$ with $r(G_\ell; q) = 2^{q^{\ell+o(1)}}$?
    \end{problem}

\vspace{2mm}
\noindent
{\bf Acknowledgements.} 
We would like to thank David Conlon for helpful comments and the anonymous referees for valuable suggestions that improved the presentation of this paper.

\end{document}